\documentclass[12pt]{amsart}
\usepackage[english]{babel}
\usepackage{paralist,url,verbatim, anysize}
\usepackage{amscd}
\usepackage[all,cmtip]{xy} 
\usepackage{amsmath,amsthm,amssymb,amsfonts}
\usepackage[bookmarksopen=true]{hyperref}
\usepackage[usenames, dvipsnames]{color}
\usepackage{graphicx}
\usepackage{hyperref}
\usepackage{euler, amsfonts, amssymb, latexsym, epsfig,epic}
\usepackage{stmaryrd}

\usepackage{tikz}
\usetikzlibrary{decorations.markings,intersections,positioning,calc}

  \tikzset{mylabel/.style  args={at #1 #2  with #3}{
    postaction={decorate,
    decoration={
      markings,
      mark= at position #1
      with  \node [#2] {#3};
 } } } }

\newcommand{\ps}[1]{\llbracket {#1} \rrbracket}

\DeclareMathOperator{\ann}{ann}
\makeatletter
\let\@fnsymbol\@arabic
\makeatother

\theoremstyle{plain}
\newtheorem{theorem}{\bf Theorem}[section]

\newtheorem{corollary}[theorem]{Corollary}
\newtheorem{lemma}[theorem]{Lemma}
\newtheorem{proposition}[theorem]{Proposition}

\newtheorem{theoremx}{Theorem}

\theoremstyle{definition}

\newtheorem{definition}[theorem]{Definition}

\newtheorem{example}[theorem]{\bf Example}
\newtheorem{remark}[theorem]{Remark}

\newtheorem*{theorem*}{\bf Theorem}

\newcommand{\Assh}{\operatorname{Assh} }

\newcommand{\ds}{\displaystyle}
\newcommand{\s}{\operatorname{s} }
\newcommand{\coker}{\operatorname{coker} }

\newcommand{\pd}{\operatorname{pd} }

\newcommand{\RR}{\mathcal{R}}

\newcommand{\FF}{\mathbb{F}}

\newcommand{\Ass}{\operatorname{Ass} }
\newcommand{\Tor}{\operatorname{Tor} }

\newcommand{\Spec}{\operatorname{Spec} }

\newcommand{\depth}{\operatorname{depth} }
\newcommand{\ee}{\operatorname{e}}

\newcommand{\p}{\mathfrak{p}}
\newcommand{\q}{\mathfrak{q}}
\newcommand{\m}{\mathfrak{m}}
\newcommand{\n}{\mathfrak{n}}

\newcommand{\NN}{\mathbb{N}}

\newcommand{\IM}{\operatorname{Im} }

\definecolor{mypink}{RGB}{215, 5, 234}

\newcommand{\rank}{\operatorname{rank}}
\newcommand{\ZZ}{\ensuremath{\mathbb{Z}}}

\newcommand{\II}{\mathbb{I}}
\newcommand{\e}{\operatorname{e}}

\newcommand{\eHK}{\operatorname{e_{HK}} }

\begin{document}
\title{Hilbert-Kunz multiplicity of powers of ideals in dimension two}
\author[A. De Stefani]{Alessandro De Stefani} 
\email{alessandro.destefani@unige.it}
\address{(A. De Stefani) Dipartimento di Matematica, Universit\`a di Genova, Via Dodecaneso 35, 16146 - Italy} 
\author[S.K. Masuti]{Shreedevi K. Masuti} 
\email{shreedevi@iitdh.ac.in}
\address{(S.K. Masuti) Department of Mathematics, Indian Institute of Technology Dharwad, Permanent Campus, Chikkamalligawad, Dharwad - 580011, Karnataka, India}
\author[M.E. Rossi]{Maria Evelina Rossi} 
\email{rossim@dima.unige.it}
\address{(M.E. Rossi) Dipartimento di Matematica, Universit\`a di Genova, Via Dodecaneso 35, 16146 - Italy} 
\author[J.K. Verma]{Jugal K. Verma} 
\email{jugal.verma@iitgn.ac.in}
\address{(J.K. Verma) IIT Gandhinagar, Palaj, Gandhinagar, 382055 Gujarat, India}


\begin{abstract} We study the behavior of the Hilbert-Kunz multiplicity of powers of an ideal in a local ring. In dimension two, we provide answers to some problems raised by Smirnov, and give a criterion to answer one of his questions in terms of a ``Ratliff-Rush version'' of the Hilbert-Kunz multiplicity.
\end{abstract}


  \date{}
\maketitle

\section{Introduction}

Let $(R,\m)$ be a Noetherian local ring of prime characteristic $p>0$ and Krull dimension $d$. The Hilbert-Kunz multiplicity of an $\m$-primary ideal $I \subseteq R$ is defined as
\[
\eHK(I) = \lim\limits_{q \to \infty} \frac{\ell_R(R/I^{[q]})}{q^d},
\]
where $q=p^e$ and $I^{[q]} = (x^{q} \mid x \in I)$ denotes the $q$-th Frobenius power of $I$. This invariant was first studied by E. Kunz \cite{Kunz69}, even though its existence as a limit was only proved several years later by P. Monsky \cite{MonskyeHK}. As the name suggests, it has similarities with the classical Hilbert-Samuel multiplicity
\[
\e(I) = \lim\limits_{n \to \infty} \frac{d! \ell_R(R/I^n)}{n^d}.
\]
Other than being defined in a similar fashion, the two invariants share some common features. For instance, if $R$ is formally unmixed, then $\e(\m) = 1$ if and only if $R$ is regular, if and only if $\eHK(\m)=1$. The first equivalence is a classical result due to P. Samuel (in the equal characteristic case \cite{Samuel}) and M. Nagata (in the general case \cite{Nagata}), while the second is a much more recent result due to K.I. Watanabe and K.I. Yoshida \cite[Theorem 1.5]{WY2} (see also \cite{HunekeYao}). Except for some similarities, however, the behavior of the two invariants is generally quite different. For instance, $\e(I)$ is always a non-negative integer, while $\eHK(I)$ can be irrational \cite{BrennerIrrational}. 

As one invariant is defined in terms of Frobenius powers $I^{[q]}$, and the other in terms of ordinary powers $I^n$, it is quite natural to wonder how the two interplay with each other. In particular, one can study the lengths $\ell_R\left(R/(I^{[q]})^n\right)$ as a function of both $q$ and $n$. For any given $q$, the classical theory of Hilbert polynomials guarantees that there exists $n_0=n_0(I,q)$ such that for all $n \geq n_0$
\[
\ell_R\left(R/(I^{[q]})^n\right) = \e_0(I^{[q]}) \binom{n+d-1}{d} - \e_1(I^{[q]})\binom{n+d-2}{d-1} + \ldots +(-1)^d \e_d(I^{[q]})
\]
for some integers $\e_i(I^{[q]})$, $i=0,\ldots,d$, called Hilbert coefficients of $I^{[q]}$. Note that $\e_0(I^{[q]}) = \e(I^{q}) = q^d\e(I)$. In \cite[Theorem 3.2]{Hanes2}, Hanes proved that 
\[
\ell_R\left(R/(I^{[q]})^n\right) = \left(\frac{\e(I)}{d!}n^d + O(n^{d-1})\right)q^d,
\]
which in particular shows that $\lim\limits_{n \to \infty} \frac{d!\eHK(I^n)}{n^d}= \e(I)$ \cite[Corollary II.7]{Hanes1}. With the goal of understanding the term $O(n^{d-1})$, V. Trivedi first (for standard graded algebras and equi-generated homogeneous ideals \cite{Trivedi}), and I. Smirnov afterwards (for any $\m$-primary ideal in a local ring \cite{Smirnov}), show that 
\[
\eHK(I^n) = \e(I) \binom{n+d-1}{d} - \lim\limits_{q \to \infty} \frac{\e_1(I^{[q]})}{q^d} \binom{n+d-2}{d-1} + o(n^{d-1}),
\]
and that the limit $\lim\limits_{q \to \infty} \frac{\e_1(I^{[q]})}{q^d}$ exists. Motivated by this result, and by the shape of the Hilbert polynomials, in \cite{Smirnov} Smirnov poses two questions and a conjecture regarding the asymptotic behavior of the function $n \mapsto \eHK(I^n)$:

\begin{enumerate}
\item[(Q1)] \label{Q1} $L_i(I):= \lim\limits_{q \to \infty} \frac{\ee_i(I^{[q]})}{q^d}$ exists as a limit for all $i=0,\ldots,d$.
\item[(Q2)] $\eHK(I^n) = \sum_{i=0}^d (-1)^iL_i(I) \binom{n+d-1-i}{d-i}$ for all $n \gg 0$.
\item[(C)] Assume further that $R$ is weakly F-regular, analytically unramified and Cohen-Macaulay. Then $L_1(I) = \ee(I)-\eHK(I)$ if and only if $I$ is stable, that is, $I^2=IJ$ for a minimal reduction $J$ of $I$.
\end{enumerate}
Conjecture (C) is actually stated in \cite{Smirnov} without the assumption that $R$ is weakly F-regular. Without this condition, it is shown in \cite{BGV} that the conjecture has actually negative answer. Note that 
(Q1) is trivial for $i=0$, and it is known for $i=1$ by \cite{Smirnov}. Therefore, (Q1) is open for $i=2,\ldots,d$. Regarding (C), we observe that if $I$ is stable then $L_1(I) = \e(I)-\eHK(I)$ is already known to hold, even without the assumption that $R$ is weakly F-regular and analytically unramified; see \cite[Theorem 1.8 (2)]{WY}.

Some progress towards answering these problems was achieved in \cite{BGV}. In Section \ref{some cases} we present more cases in which the answers are positive: ideals whose powers are of finite projective dimension in numerically Roberts rings, and parameter ideals in generalized Cohen-Macaulay rings. 
In Sections \ref{Section preliminaries} and \ref{Section main} we study the problem in bigger generality. The novelty of our approach stems from considering an asymptotic invariant defined in terms of the Ratliff-Rush closure $\widetilde{I}$ of an ideal $I$ (see Section \ref{Section preliminaries} for more details). If $I$ is $\m$-primary, we prove that the limit
\[
\widetilde{\eHK}(I) = \lim\limits_{q \to \infty} \frac{\ell_R(R/\widetilde{I^{[q]}})}{q^d},
\]
exists (see Proposition \ref{existence RR limit}) under mild assumptions, and we show that this invariant has strong connections with (Q1) and (Q2) for rings of dimension two. Our main findings can be summarized as follows:
\begin{theoremx}[Theorems \ref{main1} and \ref{main2}] Let $(R,\m)$ be a $2$-dimensional excellent Cohen-Macaulay reduced local ring, and $I$ be an $\m$-primary ideal. Then
\begin{enumerate}
\item $L_2(I)$ exists as a limit, and therefore (Q1) has a positive answer.
\item For all $n \gg 0$ we have 
\[
\widetilde{\eHK}(I^n) = \sum_{i=0}^2 (-1)^i L_i(I) \binom{n+1-i}{2-i}.
\] 
\item Conjecture (C) holds true.
\end{enumerate}
\end{theoremx}
As an immediate consequence of (2), we deduce that (Q2) has a positive answer for an ideal $I$ in a $2$-dimensional excellent Cohen-Macaulay reduced local ring if and only if $\widetilde{\eHK}(I^n) = \eHK(I^n)$ for all $n \gg 0$. We also give a characterization of this equality in terms of tight closure, see Proposition \ref{equivalence conj}. In general, one always has $\widetilde{\eHK}(I^n) \leq \eHK(I^n)$, and in Section \ref{Section main} we give an estimate of the difference between these invariants (see Proposition \ref{estimates}).

We prove that (Q2) has a positive answer in several cases.

\begin{theoremx}[Theorems \ref{thm RLR} and \ref{thm e2=0}] Let $(R,\m)$ be a two-dimensional Cohen-Macaulay local ring. 
\begin{enumerate}
\item If there exists a regular local ring $S \supseteq R$ such that the inclusion is module-finite, then (Q2) has a positive answer for every $\m$-primary ideal $I$ of $R$.
\item If $I \subseteq R$ is an $\m$-primary ideal such that $\e_2(I)=0$, then $L_2(I)=0$. Moreover, for all $n \gg 0$ we have
\[
\eHK(I^n) = \e(I) \binom{n+1}{2} - L_1(I)n.
\]
In particular, (Q2) has a positive answer for $I$.
\end{enumerate}
\end{theoremx}

Finally, in Section \ref{Section counterex} we collect some evidence that (Q2) might have negative answer, in general. We exhibit a $2$-dimensional hypersurface ring $(R,\m)$ for which we suspect that $\eHK(m^n) \ne \widetilde{\eHK}(\m^n)$ for infinitely many $n>0$. This is supported by some computational evidence, obtained by means of Macaulay2 calculations \cite{M2}.
\subsection*{Acknowledgements} Part of the research involved in this project was carried out at The Abdus Salam International Centre for Theoretical Physics (ICTP) while De Stefani, Rossi and Verma were involved in an ``ICTP-INdAM Research in Pairs'' program, and Masuti was visiting ICTP as a regular associate. We thank ICTP for the hospitality, and both ICTP and INdAM for supporting our research. We also thank Ilya Smirnov for helpful discussions on the topics of this paper, and the referee for useful suggestions and comments. De Stefani and Rossi are partially supported by the MIUR Excellence Department Project CUP D33C23001110001, and by INdAM-GNSAGA. Masuti is supported by CRG grant CRG/2022/007572 and MATRICS grant MTR/2022/000816 funded by ANRF,  Government of India. Masuti and Verma are also supported by SPARC grant SPARC/2019-2020/P1566/SL funded by Ministry of Education, Government of India.

\section{Some positive answers} \label{some cases}

We start by analyzing some special cases where the two questions (Q1),(Q2) and the conjecture (C) have positive answer. Some of them might be already known to experts, but we decided to include them here for the sake of completeness.

\subsection{Numerically Roberts rings} We recall that a Cohen-Macaulay local ring $(R,\m)$ is called {\it numerically Roberts} if $\eHK(I) = \ell_R(R/I)$ for every $\m$-primary ideal $I$ of finite projective dimension.

Examples of numerically Roberts rings include complete intersections, $2$-dimensional Cohen-Macaulay local rings, and $3$-dimensional Gorenstein local rings. For basic properties and more examples of numerically Roberts rings we refer the reader to \cite[Theorem 1.1]{Kurano}.


\begin{proposition} Let $(R,\m)$ be a numerically Roberts ring of dimension $d$, and let $I \subseteq R$ be an $\m$-primary ideal such that $\pd_R(I^n)<\infty$ for every $n >0$. Then $L_i(I) = \ee_i(I)$ for every $i=0,\ldots,d$, and $\eHK(I^n) = \sum_{i=0}^d (-1)^i L_i(I) \binom{n+d-1-i}{d-i}$. Moreover, $L_1(I) = \ee(I)-\eHK(I)$ if and only if $I$ is stable.
\end{proposition}
\begin{proof}
Let $\sigma(I)$ denote the postulation number of $I$, that is, the largest integer for which Hilbert function and Hilbert polynomial of $I$ do not coincide. By a result of C. Peskine and L. Szpiro \cite{PeskineSzpiro},  if $I^n$ has finite projective dimension, then so does $(I^n)^{[q]}$ for any $q=p^e$. For $n > \sigma(I)$ we then have 
\begin{align*}
\ell_R(R/(I^n)^{[q]}) & = \eHK((I^n)^{[q]}) \\
& = q^d \eHK(I^n) \\
& = q^d \ell_R(R/I^n) \\
& = q^d\left[\sum_{i=0}^d (-1)^i \ee_i(I) \binom{n+d-1-i}{d-i}\right] \\
& = \sum_{i=0}^d (-1)^i \left(q^d\ee_i(I)\right) \binom{n+d-1-i}{d-i}.
\end{align*}
Since, for any fixed $q$, this happens for all $n > \sigma(I)$, we have that $\ee_i(I^{[q]}) = q^d\ee_i(I)$ for every $i=0,\ldots,d$. Thus clearly $L_i(I) = \ee_i(I)$ for every $i=0,\ldots,d$. Dividing the above expression by $q^d$ and taking limits $q\to \infty$ gives that 
\[
\eHK(I^n) = \sum_{i=0}^d (-1)^i L_i(I) \binom{n+d-1-i}{d-i}.
\]
Finally, since $L_1(I)=\ee_1(I)$ and $ \eHK(I) =\ell_R(R/I)$,  the fact that $L_1(I) = \ee(I)-\eHK(I)$ is equivalent to $I$ being stable follows from the result of C. Huneke \cite{HunekeStable} and A. Ooishi \cite{OoishiStable}.
\end{proof}
Two immediate corollaries.
\begin{corollary} If $(R,\m)$ is a regular local ring, then (Q1),(Q2),(C) hold true.
\end{corollary}
\begin{corollary} Let $(R,\m,k)$ be a $d$-dimensional complete numerically Roberts ring, and let $x_1,\ldots,x_d$ be a system of parameters. Let $A=k\ps{x_1,\ldots,x_d} \subseteq R$, and $J \subseteq A$ be a $(x_1,\ldots,x_d)$-primary ideal. Then (Q1),(Q2),(C) hold true for the ideal $I=J \cdot R$.
\end{corollary}
\begin{proof} It suffices to observe that, since $R$ is Cohen-Macaulay, it is a finitely generated free $A$-module. It follows that $\pd_R(I^n) = \pd_A(J^n) < \infty$ for all $n$.
\end{proof}

\subsection{Parameter ideals in generalized Cohen-Macaulay rings} We say that a local ring $(R,\m)$ of dimension $d$ is generalized Cohen-Macaulay if $\ell_R(H^i_\m(R))< \infty$ for every $i=0,\ldots,d-1$. Note that, if $R$ is equidimensional, it is generalized Cohen-Macaulay if and only if $R_\p$ is Cohen-Macaulay for every $\p \in \Spec(R) \smallsetminus \{\m\}$. 

\begin{proposition} \label{gCM} Let $(R,\m)$ be a generalized Cohen-Macaulay local ring of dimension $d$, and $J \subseteq R$ be a parameter ideal. For every $n \geq 1$ we have $L_i(J)=0$ for all $i=1,\ldots,d$. In particular, questions (Q1) and (Q2) have positive answer for $J$.
\end{proposition}
\begin{proof}
If $d=0$ the conclusion is void, hence true. Assume that $d \geq 1$. Since $R$ is generalized Cohen-Macaulay, there exists an integer $N>0$ such that $\ee_i(J)$ is a function of $\ell_R(H^0_\m(R)),\ldots, \ell_R(H^{d-1}_\m(R))$ for any $i=1,\ldots,d$ and any parameter ideal $J \subseteq \m^N$ \cite{Schenzel}. In particular, if $J$ is a parameter ideal, then there exists $q_0$ such that $\ee_i(J^{[q]})$ is independent of $q$ for any $q \geq q_0$.

It follows that
\[
L_i(J) = \lim_{q \to \infty} \frac{\e_i(J^{[q]})}{q^d} = 0
\]
for every $i=1,\ldots,d$. Finally, by \cite[Lemma 1.3]{WY} we have that 
\[
\eHK(J^n) = \ee(J)\binom{n+d-1}{d}
\]
for every integer $n$.
\end{proof}
\begin{remark} The special case of Proposition \ref{gCM} in which $R$ is Buchsbaum was already covered by \cite[Theorem 1.3]{BGV}. Also, note that Conjecture (C) is trivial for parameter ideals.
\end{remark}

\section{Preliminaries and general methods} \label{Section preliminaries}

Let $(R,\m)$ be a $d$-dimensional Cohen-Macaulay local ring, and $I$ be an $\m$-primary ideal. A collection of ideals $\FF = \{F_n\}$ such that  $R=F_0 \supsetneq F_1 \supseteq \ldots$ is called an $I$-filtration if $IF_n \subseteq F_{n+1}$ for all $n \geq 0$. Note that $F_1$ is an $\m$-primary ideal, since it contains $I$. Also, note that we do not assume that $F_m \cdot F_n \subseteq F_{n+m}$ for $n,m \geq 0$, as we will not need to impose this condition. We say that $\FF$ is an $I$-good filtration if $IF_n = F_{n+1}$ for all $n \gg 0$. Let $J$ be a minimal reduction of $\FF$, that is, an ideal minimal with respect to containment in $I$ such that $F_{n+1} = JF_n$ for $n \gg 0$. If $R/\m$ is infinite, we may assume that $J$ is generated by a regular sequence $x_1,\ldots,x_d$. Let $r_J(\FF)$ be the reduction number of $\FF$ with respect to $J$, that is, the smallest $n$ such that $F_{n+1}=JF_n$. For an integer $n \geq 0$, let $v_{\FF}(n) = \ell_R(F_n/JF_{n-1})$. Let $H_{\FF}(n) = \ell_R(R/F_n)$ be the Hilbert-Samuel function of $\FF$, which in our assumptions for $n \gg 0$ coincides with the Hilbert Polynomial $P_{\FF}(n) = \sum_{i=0}^d (-1)^i \e_i(\FF) \binom{n+d-1-i}{d-i}$. Given  a numerical function $H:\ZZ \to \NN$ satisfying $H(n)=0$ for $n \ll 0$, we will denote by $\Delta H: \ZZ \to \NN$ the function defined as $\Delta H(n) = H(n)-H(n-1)$. For $j \geq 1$ we inductively define $\Delta^j = \Delta(\Delta^{j-1})$. If $\FF$ is a filtration as above, we will use the convention that $F_n=R$ and $\ell_R(R/F_n) = 0$ for $n<0$. In particular, we will regard $H_{\FF}$ as a function defined over $\ZZ$ by setting $H_{\FF}(n)=0$ for $n < 0$.

Let $G(\FF) = \bigoplus_{n \geq 0} F_n/F_{n+1}$. Our assumptions guarantee that $G(\FF)$ is a graded module over the associated graded ring $G(I) = \bigoplus_{n \geq 0} I^n/I^{n+1}$. We call $G(\FF)$ the associated graded module of the filtration $\FF$. If we assume that $\depth(G(\FF)) \geq d-1$, then we have some useful relations between the invariants we have defined, see \cite{HM}, \cite{GR} and \cite[Theorem 2.5]{RVBook} for the case of filtrations of modules.

\begin{proposition} \label{HM} Let $(R,\m)$ be a $d$-dimensional Cohen-Macaulay local ring, $I$ be an $\m$-primary ideal, and $\FF$ be an $I$-good filtration. Assume that $\depth(G(\FF)) \geq d-1$. Then
\begin{enumerate}
\item $\Delta^d\left(P_{\FF}(n) - H_{\FF}(n)\right) =  v_{\FF}(n)$. 
\item $\Delta^d\left(P_{\FF}(n)\right) = \e_0(\FF) = \ee_0(I)$.
\item $\Delta^d\left(H_{\FF}(n)\right) = \sum_{i=0}^d (-1)^i \binom{d}{i} \ell_R\left(R/F_{n-i}\right)$.
\item $\e_i(\FF) = \sum_{n \geq i} \binom{n-1}{i-1} v_{\FF}(n)$ for $i=1,\ldots,d$.
\end{enumerate}
\end{proposition}

We recall the definition of Ratliff-Rush closure of an ideal.

\begin{definition} Let $R$ be a Noetherian ring, and let $I \subseteq R$ be a regular ideal, that is, an ideal containing a regular element. The Ratliff-Rush closure of $I$ is defined as $\widetilde{I} = \bigcup_{n \geq 1} (I^{n+1}:I^n)$.
\end{definition}
As pointed out in \cite{RossiSwanson}, $\widetilde{(-)}$ is not a closure ``in the usual sense''; for instance, it is not always true that $I \subseteq J$ implies that $\widetilde{I} \subseteq \widetilde{J}$. Since $R$ is Noetherian, for any regular ideal $I$ there exists $n \gg 0$ such that $\widetilde{I} = I^{n+1}:I^n$. Moreover, $\widetilde{I^N} = I^N$ for any $N \gg 0$, possibly depending on $I$ (see \cite{RR}). We refer the interested reader to \cite{RossiSwanson} for more details on this operation.

Now assume that $(R,\m)$ has characteristic $p>0$. For any $q = p^e$ consider the filtration $\widetilde{\FF}_q = \{\widetilde{\FF}_{q,n}\} = \left\{\widetilde{\left(I^{[q]}\right)^n}\right\}$, which can be proved to be $I^{[q]}$-good. We also let $\FF_q = \{\FF_{q,n}\}$ with $\FF_{q,n} = \left(I^{[q]}\right)^n$. In the case of the latter, we will write $P_{I^{[q]}},H_{I^{[q]}}$ and $\ee_i(I^{[q]})$ in place of $P_{\FF_q},H_{\FF_q}$ and $\ee_i(\FF_q)$, respectively. We now collect some basic facts about $\widetilde{\FF}_q$.

\begin{proposition} \label{prop filtrations} With the above notation, for any $q=p^e$ we have that
\begin{enumerate}
\item $\depth(G(\widetilde{\FF}_q)) \geq 1$.
\item $r(\widetilde{\FF}_q) \leq r$, where $r=r_J(I)$ is the reduction number of the filtration $\{I^n\}$ with respect to a minimal reduction $J$.
\item If $\dim(R)=2$, then for all $n \geq r-1$ we have $H_{\widetilde{\FF}_q}(n) = P_{\widetilde{\FF}_q}(n)$.
\end{enumerate}
\end{proposition}
\begin{proof}
(1) follows from \cite[(2.3.1)]{RR} and \cite[(1.2)]{HLS}, and (3) is an immediate consequence of (2) and Proposition \ref{HM} (1). For (2), first observe that $J^{[q]}$ is a minimal reduction of $I^{[q]}$. Moreover, if $I^{r+1} = JI^r$, then by taking Frobenius powers we obtain that $(I^{[q]})^{r+1} = J^{[q]}(I^{[q]})^r$, and thus $r(I^{[q]}) \leq r$ for all $q$. We conclude by \cite[Proposition 4.5]{RossiSwanson} and the main result of \cite{Mafi}.
\end{proof}


In what follows, $\RR = R[It]$ will denote the Rees algebra of $I$, and $\RR' = R[It,t^{-1}]$ the extended Rees algebra. For any $q$ we let 
$\RR'_q = R[I^{[q]}t,t^{-1}]$. Let $J=(x_1,\ldots,x_d)$ be a minimal reduction of $I$, and 
set $J_q = (x_1^qt,\ldots,x_d^qt)$. 
By an analogue of the Grothendieck-Serre's formula for the Rees algebra, (see \cite{Johnston_Verma} and  \cite{Blancafort}) we have that for all $n \in \ZZ$ and all $q=p^e$ 
\begin{align*} \label{GS}
P_{I^{[q]}}(n) - H_{I^{[q]}}(n) & = \sum_{i=0}^d(-1)^i \ee_i(I^{[q]}) \binom{n+d-1-i}{d-i} - \ell_R(R/(I^{[q]})^n) \\
 & =\sum_{i=0}^d (-1)^i \ell_R(H^i_{J_q}(\RR_q'))_n).
\end{align*}

Given a $\ZZ$-graded module $M$, we let $a(M) = \sup\{n \in \ZZ \mid M_n \ne 0\}$.
\begin{proposition} \label{topLC} Let $(R,\m)$ be a $d$-dimensional local ring. Then $a(H^d_{J_q}(\RR'_q)) \leq a(H^d_{J_1}(\RR'))$ for all $q=p^e$.
\end{proposition}
\begin{proof}
Let $x_1,\ldots,x_d$ be a system of parameters in $R$, and let $n \in \ZZ$ be such that $H^d_{J_1}(\RR')_n =0$. Now let $\eta \in H^d_{J_q}(\RR'_q)_n$. From the description of local cohomology in terms of the $\check{C}$ech complex, we can write $\eta$ as the equivalence class of an element $\frac{\alpha t^{n+ds}}{(x_1^qt\cdots x_d^qt)^{s}}$ modulo the image of the localization map: 
\[
\xymatrix{
\bigoplus_{i=1}^d (\RR_q')_{(x_1^qt) \cdots \widehat{(x_i^qt)} \cdots (x_d^qt)} \ar[rrr] &&& (\RR_q')_{(x_1^qt)\cdots (x_d^qt)},
}
\]
and the fact that $\eta$ has degree $n$ means that $\alpha \in (I^{[q]})^{n+ds}$. 
Thus, we can write $\alpha = \sum_{i=1}^m r_i \alpha_i^q$ for some $\alpha_i \in I^{n+ds}$. Using again the $\check{C}$ech complex description, we can see each $\eta_i = \left[ \frac{\alpha_i t^{n+ds}}{(x_1t \cdots x_dt)^s}\right]$ as an element of $H^d_{J_1}(\RR')_n$, and thus $\eta_i=0$ by assumption. This means that there exist integers $N_i$ such that
\[
(x_1 \cdots x_d)^{N_i} \alpha_i \in \left(x_1^{N_i+s},\ldots,x_d^{N_i+s}\right)I^{(d-1)(N_i+s)+n}.
\]
Setting $N=\max\{N_i \mid i=1,\ldots,m\}$ and raising the above relations to the $q$-th Frobenius power we obtain that
\[
(x_1 \cdots x_d)^{Nq} \alpha_i^q \in \left(x_1^{q(N+s)},\ldots,x_d^{q(N+s)}\right)(I^{[q]})^{(d-1)(N+s)+n}
\]
for all $i=1,\ldots,m$. In particular, the same relation holds for $\alpha$:
\[
(x_1 \cdots x_d)^{Nq} \alpha \in \left(x_1^{q(N+s)},\ldots,x_d^{q(N+s)}\right)(I^{[q]})^{(d-1)(N+s)+n},
\]
and this precisely means that $\eta = 0$.
\end{proof}

In what follows, $F^e_*(R)$ will denote $R$ viewed as a module over itself via the $e$-th iterate of the Frobenius map $F^e:R \to R$, $F(r)=r^{p^e}$, for any integer $e \geq 1$. In case $e=1$, we will write $F$ in place of $F^1$. The ring $R$ is said to be F-finite if $F_*(R)$ is a finitely generated $R$-module. This condition is equivalent to $F^e_*(R)$ being a finitely generated $R$-module for some $e \geq 1$, and also equivalent to $F^e_*(R)$ being a finitely generated $R$-module for all $e \geq 1$.

The next lemma is a minor modification of \cite[Theorem 4.3 (i)]{PolstraTucker}.

\begin{lemma} \label{lemma estimates} Let $(R,\m)$ be an F-finite reduced local ring of dimension $d$, and $\{I_e\}_{e \in \NN}$ be a sequence of ideals satisfying: 
\begin{enumerate}[(a)]
\item There exists an $\m$-primary ideal $I$ such that $I^{[p^e]} \subseteq I_e$ for all $e$, and
\item $I_e^{[p]} \subseteq I_{e+1}$.
\end{enumerate}
Then $\lim\limits_{e \to \infty} \frac{\ell_R(R/I_e)}{p^{ed}} = \eta$ exists, and there exists a constant $C$ independent of the sequence $\{I_e\}$ such that $\eta - \frac{\ell_R(R/I_e)}{p^{ed}} \leq \frac{C}{p^e}$ for all $e \in \NN$.
\end{lemma}
\begin{proof}
Let $k=R/\m$. Let $\Assh(R) = \{\q \in \Ass(R) \mid \dim(R/\q)=d\}$. Since $R$ reduced, we have that $R_\q$ is a field for every $\q \in \Assh(R)$. Moreover, since $R$ is F-finite, for every $\q \in \Assh(R)$  by \cite[Proposition 2.3]{Kunz76} we have that $F_*(R)_\q \cong F_*(R_\q) \cong R_\q^{\oplus p^\gamma}$, where $\gamma = d+\log_p\left([F_*(k):k]\right)$. Since the modules $R^{\oplus p^\gamma}$ and $F_*(R)$ agree when localized at $W=R \smallsetminus \bigcup_{\q \in \Assh(R)} \q$, there is an exact sequence
\[
\xymatrix{
0 \ar[r] & C \ar[rr] && R^{\oplus p^{\gamma}} \ar[rr]^-{\Phi} && F_*(R) \ar[rr] && T \ar[r] & 0
}
\]
 such that $C_W = T_W = 0$. 
The assumption $I_e^{[p]} \subseteq I_{e+1}$ gives that $\Phi(I_e^{\oplus p^{\gamma}}) \subseteq I_eF_*(R) \subseteq F_*(I_{e+1})$, and thus there is an induced exact sequence
\[
\xymatrix{
\displaystyle \left(\frac{R}{I_e}\right)^{\oplus p^{\gamma}} \ar[rr]^-{\overline{\Phi}} && \displaystyle F_*\left(\frac{R}{I_{e+1}}\right) \ar[rr] &&\displaystyle \coker(\overline{\Phi}) \ar[r] &  0.
}
\]
The assumption that $I^{[p^{e+1}]} \subseteq I_{e+1}$ gives $I^{[p^e]}F_*(R) \subseteq F_*(I_{e+1})$, and thus $I^{[p^e]} \subseteq \ann_R(\coker(\overline{\Phi}))$. Since $\coker(\overline{\Phi})$ is a quotient of $T$,  we conclude that
\[
\ell_R(F_*(R/I_{e+1}))  \leq p^{\gamma} \ell_R(R/I_e) + \ell_R(\coker(\overline{\Phi})) \leq p^{\gamma} \ell_R(R/I_e) + \ell_R(T/I^{[p^e]}T).
\]
Note that $\ell_R(F_*(R/I_{e+1})) = p^{\gamma-d}\ell_R(R/I_{e+1})$, therefore dividing by $p^{\gamma +ed}$ the above estimates we get
\[
\frac{\ell_R(R/I_{e+1})}{p^{(e+1)d}} \leq \frac{\ell_R(R/I_e)}{p^{ed}} + \frac{\ell_R(T/I^{[p^e]}T)}{p^{\gamma+ed}}.
\]
Now, if $I$ is generated by $\mu$ elements, then $I^{\mu p^e} \subseteq I^{[p^e]}$, and thus
\[
\ell_R(T/I^{[p^e]}T) \leq \ell_R(T/I^{\mu p^e}T) \leq D\mu^{\dim(T)} p^{e\dim(T)}
\]
for some $D$ which only depends on $T$ and $I$, since $\ell_R(T/I^{\mu p^e}T)$ is eventually a polynomial of degree $\dim(T)$ in the variable $\mu p^e$. Note that, since $T_\q=0$ for all $\q \in \Assh(R)$, we have that $\dim(T)\leq d-1$. Set $C=D\mu^{d-1}p^{-\gamma}$, and note that it is independent of the sequence $\{I_e\}$. From the above estimates we conclude that
\[
\frac{\ell_R(R/I_{e+1})}{p^{(e+1)d}} \leq \frac{\ell_R(R/I_e)}{p^{ed}} + \frac{C}{p^e},
\]
and the lemma follows from \cite[Lemma 3.5 (i)]{PolstraTucker}.
\end{proof}

As a consequence of Lemma \ref{lemma estimates} we obtain the existence of the limit of Ratliff-Rush closures of Frobenius powers of an ideal:

\begin{proposition} \label{existence RR limit}
Let $(R,\m)$ be an excellent reduced local ring of dimension $d>0$. If $I \subseteq R$ is an $\m$-primary ideal, then the following limit exists
\[
\widetilde{\eHK}(I;R) := \lim\limits_{q \to \infty} \frac{\ell_R(R/\widetilde{I^{[q]}})}{q^d}.
\]
\end{proposition}
\begin{proof}
Since $R$ is excellent, its completion is reduced. Moreover, we have that $\widetilde{I} \ \widehat{R} = \widetilde{I\widehat{R}}$ since $R \to \widehat{R}$ is faithfully flat \cite[(1.7)]{HLS}. Thus, we may assume that $R$ is complete.
By \cite[Lemma 6.13]{HHSmooth}, there exists an $F$-finite reduced ring $S$ such that $R \to S$ is faithfully flat, purely inseparable, and $\m S$ is the maximal ideal of $S$. Such a ring is obtained from $R$ via the ``so-called'' $\Gamma$ construction. Since $R\to S$ is faithfully flat, if $J$ is an ideal of $R$ containing a regular element we have that $\widetilde{J}S = \widetilde{JS}$ \cite[(1.7)]{HLS}. Moreover, the condition that $\m S$ is the maximal ideal of $S$ guarantees that $\ell_R(R/J) = \ell_S(S/JS)$ for any $\m$-primary ideal $J \subseteq R$. Since $R$ is reduced, $(I^{[q]})^n$ contains a regular element for every $q=p^e$ and every $n>0$. From all these facts, it follows that
\[
\ell_R\left(\frac{R}{\widetilde{(I^{[q]})^n}}\right) = \ell_{S}\left(\frac{S}{\widetilde{(I^{[q]})^n}S}\right) = \ell_{S}\left(\frac{S}{\widetilde{((IS)^{[q]})^n}}\right).
\]
By replacing $R$ with $S$, we may therefore assume that $R$ is reduced and F-finite. It now suffices to show that the sequence $\{I_e\}$ with $I_e = \widetilde{I^{[p^e]}}$ satisfies the assumptions of Lemma \ref{lemma estimates}: if $t \in \NN$ is such that $\widetilde{I^{[p^e]}} = (I^{[p^e]})^{t+1}:_R (I^{[p^e]})^{t}$, then 
\[
\widetilde{I^{[p^e]}}^{[p]} = \left[(I^{[p^e]})^{t+1}:_R (I^{[p^e]})^{t}\right]^{[p]} \subseteq (I^{[p^{e+1}]})^{t+1}:_R (I^{[p^{e+1}]})^{t} \subseteq \widetilde{I^{[p^{e+1}]}}.
\]
Moreover, $I^{[p^e]} \subseteq \widetilde{I^{[p^e]}}$ for all $e$.
\end{proof}

If no confusion on the ambient ring may arise, we will denote $\widetilde{\eHK}(I;R)$ simply by $\widetilde{\eHK}(I)$. 

\begin{remark} We point out that $\widetilde{\eHK}(I)$ may differ from $\eHK(I)$ in general. In fact, let $(R,\m)$ be a regular local ring of dimension $d$, and $I \subseteq R$ be an $\m$-primary ideal such that $I \ne \widetilde{I}$. Let $t>0$ be such that $\widetilde{I} = I^{t+1}:I^t$. By flatness of Frobenius on $R$, for all $q=p^e$ we have that 
\[
\widetilde{I^{[q]}} \supseteq (I^{[q]})^{t+1}:(I^{[q]})^t = (I^{t+1}:I^t)^{[q]} = \widetilde{I}^{[q]}.
\]
It follows that
\[
\ell_R\left(\frac{R}{\widetilde{I^{[q]}}}\right) \leq \ell_R\left(\frac{R}{\widetilde{I}^{[q]}}\right) = q^d \ell_R\left(\frac{R}{\widetilde{I}}\right) < q^d \ell_R\left(\frac{R}{I}\right) = \ell_R\left(\frac{R}{I^{[q]}}\right)
\]
and, in particular, $\widetilde{\eHK}(I) < \eHK(I)$. An explicit example of such an ideal in dimension two is $R=k\ps{x,y}$, with $k$ a field, and $I=(x^4,x^3y,xy^3,y^4)$. In fact, in this case $x^2y^2 \in \widetilde{I} \smallsetminus I$. 
\end{remark}

An alternative proof of the existence of $\widetilde{\eHK}(I)$ for excellent reduced local rings can be obtained using uniform estimates as in \cite[Theorem 3.6]{TuckerFsignature} or \cite[Corollary 5.5]{SmirnovICTP}. However, with the approach of Lemma \ref{lemma estimates} based on \cite{PolstraTucker} one can get further information on the convergence of $\ell_R(R/\widetilde{I^{[q]}})/q^d$ to its limit $\widetilde{\eHK}(I)$. To this end, we record another slight variation of a result of Polstra and Tucker:

\begin{proposition} Let $(R,\m)$ be an F-finite reduced local ring of dimension $d$, and $\{I_{t,e}\}_{t,e \in \NN}$ be a sequence of ideals such that:
\begin{enumerate}[(a)]
\item $I_{t,e}^{[p]} \subseteq I_{t,e+1}$ for all $t$ and $e$,
\item $I_{t,e} \subseteq I_{t+1,e}$ for all $t$ and $e$,
\item there exists an $\m$-primary ideal $I$ such that $I^{[p^e]} \subseteq I_{t,e}$ for all $t$ and $e$.
\end{enumerate}
If we set $I_e=\sum_{t \in \NN} I_{t,e}$, then $\lim\limits_{e\to \infty} \frac{\ell_R(R/I_e)}{p^{ed}}$ exists as a limit, and
\[
\lim\limits_{t \to \infty} \lim\limits_{e \to \infty} \frac{\ell_R(R/I_{t,e})}{p^{ed}} = \lim\limits_{e\to \infty} \frac{\ell_R(R/I_e)}{p^{ed}} = 
\lim\limits_{e \to \infty} \lim\limits_{t \to \infty} \frac{\ell_R(R/I_{t,e})}{p^{ed}}.
\]
\end{proposition}
\begin{proof}
The proof is analogous to that of \cite[Theorem 6.3]{PolstraTucker}, using Lemma \ref{lemma estimates} in place of \cite[Theorem 4.3 (i)]{PolstraTucker}.
\end{proof}

\begin{corollary} Let $(R,\m)$ be an F-finite reduced local ring of dimension $d>0$, and $I \subseteq R$ be an $\m$-primary ideal. Then
\[
\widetilde{\eHK}(I)= \lim_{t \to \infty} \lim_{e \to \infty} \frac{\ell_R(R/I_{t,e})}{p^{ed}},
\]
where $I_{t,e} = (I^{[p^e]})^{t+1}:_R(I^{[p^e]})^t$.
\end{corollary}

We show that a transformation rule along finite map holds for $\widetilde{\eHK}$. 

\begin{proposition} \label{transformation rule} Let $(R,\m) \subseteq (S,\n)$ be a finite extension of excellent local domains of dimension $d>0$, and $I \subseteq R$ be an $\m$-primary ideal. Then 
\[
\widetilde{\eHK}(I;R) = \frac{[S/\n:R/\m]}{\rank_R(S)} \lim\limits_{q \to \infty} \frac{\ell_S(S/\widetilde{I^{[q]}}S)}{q^d}.
\]
In addition, if $S$ is flat over $R$, then
\[
\widetilde{\eHK}(I;R) = \frac{[S/\n:R/\m]}{\rank_R(S)} \widetilde{\eHK}(IS;S).
\]
\end{proposition}
\begin{proof}
Let $r=\rank_R(S)$, and $s=[S/\n:R/\m]$. Since $S$, as an $R$-module, is isomorphic to $R^{\oplus r}$ when localized at $R \smallsetminus \{0\}$, we have short exact sequences $0 \to R^{\oplus r} \to S \to C_1 \to 0$ and $0 \to S \to R^{\oplus r} \to C_2 \to 0$, with $\dim(C_1)<d$ and $\dim(C_2) < d$. The exactness on the left is due to the fact that both $R^{\oplus r}$ and $S$ are torsion-free $R$-modules. Tensoring the first sequence with $R/\widetilde{I^{[q]}}$ and counting lengths we obtain
\begin{align*}
\ell_R\left(S/\widetilde{I^{[q]}}S\right) - r\ell_R\left(R/\widetilde{I^{[q]}}\right) & \leq \ell_R(C_1/\widetilde{I^{[q]}}C_1) \\
& \leq \ell_R(C_1/I^{[q]}C_1) \\
& \leq \ell_R(C_1/I^{\mu(I)q}C_1),
\end{align*}
where we use that $I^{\mu(I)q} \subseteq I^{[q]}$ by the pigeonhole principle and that $I^{[q]} \subseteq \widetilde{I^{[q]}}$. Since the function $n \mapsto \ell_R(C_1/I^{n}C_1)$ is eventually a polynomial of degree $\dim(C_1)$, we can find $D_1>0$ such that $\ell_R(C_1/I^{\mu(I)q}C_1) \leq D_1(\mu(I)q)^{\dim(C_1)} \leq E_1 q^{d-1}$, where $E_1 := D_1\mu(I)^{d-1}$ is a constant independent of $q$. A similar argument applied to the second short exact sequence gives the existence of a constant $E_2>0$ independent of $q$ such that
\[
r\ell_R\left(R/\widetilde{I^{[q]}}\right)-\ell_R\left(S/\widetilde{I^{[q]}}S\right) \leq E_2q^{d-1}.
\]
By setting $E:=\max\{E_1,E_2\}$ we therefore get that
\[
\left|\ell_R\left(S/\widetilde{I^{[q]}}S\right) - r\ell_R\left(R/\widetilde{I^{[q]}}\right)\right| \leq Eq^{d-1}
\]
for some constant $E>0$ independent of $q$. Dividing by $q^d$ and taking limits gives that 
\[
r\widetilde{\eHK}(I;R) = \lim\limits_{q \to \infty} \frac{\ell_R(S/\widetilde{I^{[q]}}S)}{q^d}.
\] 
Finally, note that $\ell_R(S/\widetilde{I^{[q]}}S) = s \ell_S(S/\widetilde{I^{[q]}}S)$, and the first claimed equality is now proved. If $R\to S$ is flat, we have that $\widetilde{I^{[q]}}S = \widetilde{I^{[q]}S} = \widetilde{(IS)^{[q]}}$ by \cite[(1.7)]{HLS}. Note that $I$ contains a regular element since $R$ is a domain of positive dimension and $I$ is $\m$-primary. The second claimed equality now follows as well.
\end{proof}

By definition, we have $\widetilde{\eHK}(I) \leq \eHK(I)$ for any $\m$-primary ideal $I$. We now give a characterization of the equality $\widetilde{\eHK}(I) = \eHK(I)$ in terms of tight closure, in the spirit of \cite[Theorem 8.17]{HHTightClosure}. Let $R^\circ$ be the complement of the union of the minimal primes of $R$. We recall that, given an ideal $I \subseteq R$, its tight closure is
\[
I^* = \left\{x \in R \mid \text{ there exists } c \in R^\circ \text{ such that } cx^q \in I^{[q]} \text{ for all } q=p^e \gg 0\right\}.
\]
Recall that a local ring $(R,\m)$ is said to be formally equidimensional if its $\m$-adic completion $\widehat{R}$ is equidimensional.
\begin{proposition} \label{equivalence conj}
Let $(R,\m)$ be an excellent, formally equidimensional, reduced local ring of dimension $d>0$, and $I$ be an $\m$-primary ideal. Then $\widetilde{\eHK}(I) = \eHK(I)$ if and only $\widetilde{I^{[q]}} \subseteq (I^{[q]})^*$ for all $q=p^e$.
\end{proposition}
\begin{proof}
First assume that $\widetilde{I^{[q]}} \subseteq (I^{[q]})^*$ for all $q=p^e$. By \cite[Theorem 6.1]{HHSmooth} there exists a test element $c \in R^\circ$, that is, an element $c$ such that $\dim(R/(c)) < d$ and $cJ^* \subseteq J$ for any ideal $J$ of $R$. Consider the exact sequence
\[
\xymatrix{
0 \ar[r] & \ds \frac{I^{[q]}:_R(c)}{I^{[q]}} \ar[r] & \ds \frac{R}{I^{[q]}} \ar[r] & \ds \frac{R}{I^{[q]}} \ar[r] & \ds \frac{R}{(I^{[q]},c)} \ar[r] & 0.
}
\]
Since all modules involved have finite length, we have that the first and last ones have the same length. Moreover, we have that $c\widetilde{I^{[q]}} \subseteq c(I^{[q]})^* \subseteq I^{[q]}$, so that
\[
0 \leq \ell_R\left(\frac{\widetilde{I^{[q]}}}{I^{[q]}}\right)  \leq \ell_R\left(\frac{I^{[q]}:_R c}{I^{[q]}}\right) = \ell_R\left(\frac{R}{(I^{[q]},c)}\right) \leq D q^{d-1}
\]
for a constant $D>0$ which only depends on $I$ and $c$. Dividing by $q^d$ and taking limits gives that $\widetilde{\eHK}(I) = \eHK(I)$.

Conversely, suppose that there exists $q_0$ such that $\widetilde{I^{[q_0]}} \not\subseteq (I^{[q_0]})^*$. Since the first ideal always contains $I^{[q_0]}$, by \cite[Theorem 8.17]{HHTightClosure} we have that $\eHK(I^{[q_0]}) > \eHK(\widetilde{I^{[q_0]}})$. But then, since $\left(\widetilde{I^{[q_0]}}\right)^{[q]} \subseteq \widetilde{I^{[q_0q]}}$ holds for all $q$, we conclude that
\[
0< \lim\limits_{q \to \infty} \frac{\ell_R\left((\widetilde{I^{[q_0]}})^{[q]}/I^{[q_0q]}\right)}{q^d} \leq \lim\limits_{q \to \infty} \frac{\ell_R\left(\widetilde{I^{[q_0q]}}/I^{[q_0q]}\right)}{q^d} = q_0^d\left(\eHK(I) - \widetilde{\eHK}(I)\right). \qedhere
\]
\end{proof}

\section{Results in dimension two} \label{Section main}
Throughout this section, unless otherwise stated, we will assume that $(R,\m)$ is a $2$-dimensional Cohen-Macaulay local ring. Without loss of generality, by replacing $R$ with $R[x]_{\m R[x]}$ if needed, we may also assume that $R/\m$ is infinite. Throughout, we will use the convention that any non-positive power of an ideal is the unit ideal. The following is our main result. 

\begin{theorem} \label{main1}
Let $(R,\m)$ be a $2$-dimensional excellent Cohen-Macaulay reduced local ring, and $I$ be an $\m$-primary ideal. Let $r=r_J(I)   $ be the reduction number of $I$ with respect to a minimal reduction $J$. With the above notation we have
\begin{enumerate}
\item $L_1(I) = r \e(I) + \widetilde{\eHK}(I^{r-1}) - \widetilde{\eHK}(I^r)$.
\item $L_2(I) = \binom{r}{2} e(I) + r \ \widetilde{\eHK}(I^{r-1})   -(r-1) \widetilde{\eHK}(I^r) $. In particular, $L_2(I)$ exists as a limit and (Q1) has a positive answer.
\item The equality $\widetilde{\eHK}(I^n) = \sum_{i=0}^2 (-1)^i L_i(I) \binom{n+1-i}{2-i}$ holds for all $n \geq r-1$. In particular, (Q2) has a positive answer for $I$ if and only if $\widetilde{\eHK}(I^n) = \eHK(I^n)$ for all $n \gg 0$.
\end{enumerate}
\end{theorem}
\begin{proof}
Let $\widetilde{\FF}_q = \{ \widetilde{(I^{[q]})^n}\}$, which is an $I^{[q]}$-good filtration as already observed. We use the notation introduced in Section \ref{Section preliminaries}. By Proposition \ref{HM} and \cite[Theorem 2.1]{RR} we have
\[
\e_i(I^{[q]}) = \e_i(\widetilde{\FF}_q) = \sum_{n=i}^r \binom{n-1}{i-1} v_{\widetilde{\FF}_q}(n)
\]
for $i=1,2$. Moreover, we get
\begin{align*}
v_{\widetilde{\FF}_q}(n) & = \Delta^2\left(P_{\widetilde{\FF}_q}(n) - H_{\widetilde{\FF}_q}(n)\right) \\
& = \e(I^{[q]}) - \Delta^2\left(H_{\widetilde{\FF}_q}(n)\right) \\
& = q^2\e(I) - \left[\ell_R\left(\frac{R}{\widetilde{\left(I^{[q]}\right)^n}}\right)  - 2\ell_R\left(\frac{R}{\widetilde{\left(I^{[q]}\right)^{n-1}}}\right) + \ell_R\left(\frac{R}{\widetilde{\left(I^{[q]}\right)^{n-2}}}\right)  \right].
\end{align*}
Replacing in the above equality, for $i=1$ we get that
\[
\e_1(I^{[q]}) = rq^2\e(I) + \ell_R\left(\frac{R}{\widetilde{\left(I^{[q]}\right)^{r-1}}}\right) - \ell_R\left(\frac{R}{\widetilde{\left(I^{[q]}\right)^r}}\right).
\]
Dividing by $q^2$, taking $\lim_{q \to \infty}$ and using Proposition \ref{existence RR limit} we get the desired statement for $L_1(I)$. For $i=2$ the proof is analogous; in particular the limit $L_2(I)$ exists thanks to Proposition \ref{existence RR limit}.

For (3): by Proposition \ref{prop filtrations} we have that $r(\widetilde{\FF}_q) \leq r$, and for all $n \geq r-1$ we have
\[
\ell_R\left(\frac{R}{\widetilde{\left(I^{[q]}\right)^n}}\right)  = H_{\widetilde{\FF}_q}(n) = P_{\widetilde{\FF}_q}(n) = \sum_{i=0}^2 (-1)^i \e_i(I^{[q]}) \binom{n+1-i}{2-i}.
\]
Dividing by $q^2$, taking limits and Proposition \ref{existence RR limit} gives the desired statement.
\end{proof}

We record an immediate consequence of Theorem \ref{main1} and Proposition \ref{equivalence conj}.

\begin{corollary} \label{coroll equiv} Let $(R,\m)$ be a $2$-dimensional excellent Cohen-Macaulay reduced local ring, and $I$ be an $\m$-primary ideal. The following are equivalent:
\begin{enumerate}
\item (Q2) has a positive answer for $I$;
\item $\eHK(I^n) = \widetilde{\eHK}(I^n)$ for all $n \gg 0$.
\item $\widetilde{(I^{[q]})^n)} \subseteq ((I^{[q]})^n)^*$ for all $q=p^e$ and all $n \gg 0$.
\end{enumerate}
\end{corollary}

Thanks to these equivalences, we obtain that (Q2) has positive answer for a large class of rings: finite subrings of $2$-dimensional regular local rings. The argument is the outcome of some useful discussions with Smirnov. We thank him for allowing us to include this result here. 


\begin{theorem} \label{thm RLR} Let $(R,\m)$ be a $2$-dimensional excellent Cohen-Macaulay local ring, and suppose that there exists a regular local ring $S \supseteq R$ such that the inclusion is module-finite. Then (Q2) has a positive answer for any $\m$-primary ideal $I$ of $R$.
\end{theorem}
\begin{proof}
Let $I \subseteq R$ be an $\m$-primary ideal. 
Let $J =IS$, and choose $n_0$ such that $\widetilde{J^n} = J^n$ for all $n \geq n_0$. 
Let $n \geq n_0$ and $q=p^e$. Since $S$ is regular, for $t \gg 0$ we have
\[
\widetilde{(J^{[q]})^n} = (J^{[q]})^{n+t}:(J^{[q]})^t =  \left(J^{n+t}:J^t\right)^{[q]}  \subseteq (\widetilde{J^n})^{[q]} = (J^n)^{[q]} = (J^{[q]})^n.
\]
In particular, we get that
\[
\left(\widetilde{(I^{[q]})^n}\right)S \subseteq \widetilde{(I^{[q]})^nS} = \widetilde{(J^{[q]})^n} = (J^{[q]})^n = ((I^{[q]})^n)S.
\]
Thus, we have $\left(\widetilde{(I^{[q]})^n}\right)S = ((I^{[q]})^n)S$ for all $q=p^e$, and by Proposition \ref{transformation rule} we then get that 
\[
\widetilde{\eHK}(I^n;R) = \frac{[S/\n:R/\m]}{\rank_R(S)} \lim\limits_{q \to \infty} \frac{\ell_S(S/(J^{[q]})^n)}{q^d} = \frac{[S/\n:R/\m]}{\rank_R(S)} \eHK(J^n;S).
\]
On the other hand, the transformation rule for Hilbert-Kunz multiplicities along finite maps (for instance, see \cite[Theorem 3.16]{HunekeSurvey}) guarantees that 
\[
\eHK(I^n;R) = \frac{[S/\n:R/\m]}{\rank_R(S)} \eHK(J^n;S),
\]
and it follows that $\eHK(I^n;R) = \widetilde{\eHK}(I^n;R)$ for all $n \geq n_0$. Since $R$ is a subring of a regular local ring, it is reduced. Thus, we  conclude thanks to Corollary \ref{coroll equiv}.
\end{proof} 


We recall the definition of F-signature for a ring of any dimension. Let $\alpha = \log_p \left([F_*(k):k]\right)$, where $k=R/\m.$ 
\begin{definition}
Let $(R,\m)$ be an F-finite local ring of dimension $d$. For every $e>0$ write $F^e_*(R) = R^{\oplus a_e} \oplus M_e$, where $M_e$ is an $R$-module with no free summands. The F-signature of $R$ is
\[
\s(R) = \lim\limits_{e \to \infty} \frac{a_e}{p^{e(d+\alpha)}}.
\]
\end{definition}
The existence of the limit was proved in full generality by K. Tucker \cite{TuckerFsignature}.

For a $2$-dimensional local ring $R$ we let $\underline{L}_2(I) = \liminf\limits_{q \to \infty} \frac{e_2(I^{[q]})}{q^2}$ and $\overline{L}_2(I) = \limsup\limits_{q \to \infty} \frac{e_2(I^{[q]})}{q^2}$.

\begin{proposition} \label{prop inequalities e2} Let $(R,\m)$ be a $2$-dimensional F-finite Cohen-Macaulay local ring. For all $q=p^e$ we have $\frac{a_e}{q^{\alpha}} \ee_2(I) \leq \ee_2(I^{[q]}) \leq \ee_2(I) \ell_R(R/\m^{[q]})$. In particular,
\[
\ee_2(I)\s(R) \leq \underline{L}_2(I) \leq  \overline{L}_2(I) \leq \ee_2(I)\eHK(\m).
\]
\end{proposition}
\begin{proof}
Let $J=(x,y)$ be a minimal reduction of $I$, and for $s \in \NN$ set $A_s = I^{2s}$ and $B_s= J^{[s]}I^s$, where $J^{[s]}:=(x^s,y^s)$. As in Section \ref{Section preliminaries}, we denote by $\RR'_q$ the extended Rees algebra of $I^{[q]}$, and $J_q = (x^qt,y^qt)$. For every $q=p^e$, by the Grothendieck-Serre's formula evaluated at $n=0$ (see \cite{Blancafort,Johnston_Verma}) and by \cite[Lemma 2.2]{Sally}, we have that
\[
\ee_2(I^{[q]}) = \ell_R(H^2_{J_q}(\RR'_q)_0) = \ell_R\left(\frac{A_s^{[q]}}{B_s^{[q]}}\right)
\]
for all $s \gg 0. $ For a fixed $q=p^e$ choose $s \gg 0$ so that $\ell_R(A_s/B_s) = \ee_2(I)$ and $\ell_R(A_s^{[q]}/B_s^{[q]}) = \ee_2(I^{[q]})$. Then 
\[
\ee_2(I^{[q]}) = \ell_R\left(\frac{A_s^{[q]}}{B_s^{[q]}}\right) = \frac{1}{q^{\alpha}} \ell_R\left(\frac{A_sF^e_*(R)}{B_sF^e_*(R)}\right) \geq \frac{a_e}{q^{\alpha}} \ell_R\left(\frac{A_s}{B_s}\right) =\frac{a_e}{q^{\alpha}} \ee_2(I).
\]
Dividing by $q^2$ and taking $\liminf$ gives
\[
\underline{L}_2(I) \geq \ee_2(I) \lim\limits_{e \to \infty} \frac{a_e}{q^{2+\alpha}} = \ee_2(I) \s(R).
\]
For the upper bound, as above let $s$ be such that $\ell_R(A_s/B_s) = \ee_2(I)$ and $\ell_R(A_s^{[q]}/B_s^{[q]}) = \ee_2(I^{[q]})$. Let $B_s = L_0 \subsetneq L_1 \subsetneq \ldots \subsetneq L_n = A_s$ with $L_{j+1}/L_j \cong R/\m$ be a composition series of $A_s/B_s$. Since $\m L_{j+1} \subseteq L_j$ we have that $\m^{[q]} L_{j+1}^{[q]} \subseteq L_j^{[q]}$, that is, we have a surjection $R/\m^{[q]} \twoheadrightarrow L_{j+1}^{[q]}/L_j^{[q]}$. It follows that 
\[
\ee_2(I^{[q]}) = \ell_R\left(\frac{A_s^{[q]}}{B_s^{[q]}}\right) \leq \ell_R\left(\frac{A_s}{B_s}\right) \cdot \ell_R\left(\frac{R}{\m^{[q]}}\right) = \ee_2(I) \ell_R\left(\frac{R}{\m^{[q]}}\right).
\]
Dividing by $q^2$ and taking $\limsup$ gives the desired statement.
\end{proof}

\begin{proposition} \label{equivalence e2=0} Let $(R,\m)$ be a $2$-dimensional F-finite Cohen-Macaulay local ring, and consider the following conditions:
\begin{enumerate}
\item $\ee_2(I) = 0$.
\item $\ee_2(I^{[q]}) = 0$ for all $q=p^e$.
\item $L_2(I)$ exists as a limit, and $L_2(I) = 0$.
\item $\underline{L}_2(I) = 0$
\end{enumerate}
Then (1) $\Rightarrow$ (2) $\Rightarrow$ (3) $\Rightarrow$ (4). In addition, (1) $\iff$ (2) if $R$ is F-pure, and they are all equivalent if $R$ is weakly F-regular.
\end{proposition}
\begin{proof}
First, recall $\ee_2(I) \geq 0$ by \cite{Narita}. Moreover, recall that $R$ is F-pure if and only if $a_e>0$ for some (equivalently, for all) $e>0$. In dimension two $R$ is weakly F-regular if and only if it is strongly F-regular \cite{Williams}, and the latter is equivalent to $\s(R)>0$ in any dimension by \cite{AberbachLeuschke}. The statement is now an immediate consequence of Proposition \ref{prop inequalities e2}.
\end{proof}

\begin{remark} \label{remark Hoa} If $(R,\m)$ is a $2$-dimensional local ring, and $I$ is an $\m$-primary ideal, then by \cite{Hoa} there exists $n_0$ such that for $n \geq n_0$ the reduction number of $I^n$ is independent of the choice of a minimal reduction, and it equals either $1$ or $2$. 
If $R$ is Cohen-Macaulay it can be shown that $\e_2(I)=0$ if and only if such a value equals $1$.
\end{remark}

\begin{theorem} \label{thm e2=0}
Let $(R,\m)$ be a $2$-dimensional Cohen-Macaulay local ring. Assume that $\ee_2(I)=0$. Then for all $n\gg 0$ we have
\[
\eHK(I^n) = \ee(I) \binom{n+1}{2} - L_1(I) n.
\]
In particular, (Q2) has a positive answer for $I$. 
\end{theorem}
\begin{proof}
First, we reduce to the case in which $R$ is F-finite. Let $k$ be the residue field of $R$. By Cohen's structure theorem, the completion $\widehat{R}$ of $R$ at $\m$ contains an isomorphic copy of $k$. Let $\overline{k}$ denote an algebraic closure of $k$, and let $S$ denote the completion of $\widehat{R} \otimes_k \overline{k}$ at the ideal obtained by extending $\m$. Then $(S,\n)$ is a complete local ring with residue field $S/\n \cong \overline{k}$. In particular, it is an F-finite ring. Moreover, the map $(R,\m) \to (S,\n)$ is faithfully flat with $\m S = \n$, therefore base changing to $S$ does not affect calculations of length, multiplicity or Hilbert coefficients. Thus, after replacing $R$ with $S$, we may directly assume that $R$ is F-finite. Let $J=(x,y)$ be a minimal reduction of $I$. For all $n > 0$ let $J^{[n]} = (x^n,y^n)$. By \cite[Lemma 2.2]{Sally}, the condition $\ee_2(I)=0$ gives that $I^{2n} = J^{[n]}I^n$ for all $n \geq n_0$, for some $n_0$. In particular $I^n$ is a stable ideal and, as already pointed out, this gives that $L_1(I^n) = e(I^n)-\eHK(I^n)$ for all $n \geq n_0$ (see  \cite[Theorem 1.8 (2)]{WY}, or \cite{Smirnov}). For any ideal $K$ one has $\ee_1(K^n) = n\ee_1(K) + \binom{n}{2}\ee(K)$, and it follows that $L_1(I^n) = nL_1(I) + \binom{n}{2}\ee(I)$. Since $\ee(I^n)=n^2\ee(I)$, a direct calculation shows that $\eHK(I^n) = \ee(I)\binom{n+1}{2} - L_1(I)n$ for all $n \geq n_0$. To conclude the proof, it suffices to note that $L_2(I)=0$ by Proposition \ref{equivalence e2=0}. 
\end{proof}


We now turn our attention to Conjecture (C). We prove that it holds true for local rings of dimension two.

\begin{theorem} \label{main2}
Let $(R,\m)$ be an excellent $2$-dimensional analytically unramified weakly F-regular local ring of characteristic $p>0$. Then (C) holds true.
\end{theorem}
\begin{proof}
We only need to show that $L_1(I)=\ee(I) - \eHK(I)$ implies that $I$ is stable. By Proposition \ref{topLC} there exists $n_0$ such that $H^2_{J_q}(\RR_q')_n = 0$ for all $n \geq n_0$ and all $q=p^e$. For any such $n$ and all $q$ we have that
\[
\ee(I^{[q]})\binom{n+1}{2} - \ee_1(I^{[q]})n + \ee_2(I^{[q]}) - \ell_R(R/(I^{[q]})^n) = - \ell_R(H^1_{J_q}(\RR'_q)_n) \leq 0.
\]
Since $R$ is reduced, dividing by $q^2$ and taking limits we conclude that $\eHK(I^n) \geq e(I)\binom{n+1}{2} - L_1(I)n + L_2(I)$ by Theorem \ref{main1}. Since $\ee_2(K) \geq 0$ for any $\m$-primary ideal $K$ \cite{Narita}, we have that $L_2(I) \geq 0$. By \cite[Theorem 1.8 (1)]{WY}, we have that $\eHK(I^n) \leq e(I)\binom{n+1}{2} - (\ee(I)-\eHK(I))n$. Putting these facts together and canceling out terms we get that $L_1(I)n - L_2(I) \geq (\ee(I) - \eHK(I))n$. If we assume that $L_1(I) = \ee(I) - \eHK(I)$, then from $L_2(I) \geq 0$ we conclude that $L_2(I) = 0$, and that $\eHK(I^n) = \e(I)\binom{n+1}{2} - (\e(I)-\eHK(I))n$. The claim now follows from \cite[Theorem 1.8 (4)]{WY}.
\end{proof}

\begin{remark} We note that the assumptions of Theorem \ref{main2} guarantee that $R$ is Cohen-Macaulay. In fact, by \cite{Kawasaki} we have that excellent rings are the homomorphic image of a Cohen-Macaulay local ring. This condition, together with the fact that $R$ is weakly F-regular, implies that $R$ is Cohen-Macaulay by \cite[Theorem (3.4) (c)]{HHSmooth}.
\end{remark}

\begin{remark} For normal ideals in a pseudo-rational $2$-dimensional local ring, (Q1),(Q2) and (C) trivially hold true. Indeed, in this case $I$ is stable by \cite{Rees}. Note that by Proposition \ref{equivalence e2=0} and Remark \ref{remark Hoa} we have that $L_2(I)=0$ in this case.
\end{remark}

Recall that, from \cite{Smirnov}, for $n \gg 0$ we have that $\eHK(I^n) = \ee(I) \binom{n+1}{2} - L_1n + f(n)$ for some function $f(n)$ such that $\lim_{n \to \infty} \frac{f(n)}{n} =0$. Our next goal is to show that $f(n)$ is eventually non-increasing.

\begin{lemma} \label{lemma exact sequence} Let $(R,\m)$ be a $2$-dimensional Cohen-Macaulay local ring, and $J=(x,y)$ be a parameter ideal. For any $t>0$ and any ideal $I \subseteq R$ we have an exact sequence
\[
\xymatrix{
0 \ar[r] & \Tor_1^R(R/I,J^t) \ar[r] & \displaystyle \left(\frac{R}{I}\right)^{\oplus t} \ar[r]^-{\varphi} & \displaystyle \left(\frac{R}{I}\right)^{\oplus (t+1)} \ar[r] &  \displaystyle\frac{J^t}{IJ^t} \ar[r] & 0,
}
\]
where $\IM(\varphi) \subseteq J \left(R/I\right)^{\oplus (t+1)}$. In particular, $ \displaystyle \left((I:J)/I\right)^{\oplus t} \subseteq \ker(\varphi)$.
\end{lemma}
\begin{proof}
As $x,y$ is a regular sequence, a minimal free resolution of $J^t$ is given by the Eagon-Northcott complex (or by Hilbert-Burch, in this case):
\[
\xymatrix{
0 \ar[r] & R^{\oplus t} \ar[rrrr]^-{\begin{bmatrix}
x & 0 & 0 \ldots & 0  \\
-y & x & 0 \ldots & 0\\
\vdots & \vdots & \vdots & \vdots 
\end{bmatrix}} &&&& R^{\oplus (t+1)} \ar[r] & J^t \ar[r] & 0.
}
\]
Tensoring the above with $R/I$ gives the desired exact sequence, together with the containment $\IM(\varphi) \subseteq J \left(R/I\right)^{\oplus (t+1)}$. It is therefore clear that $ \displaystyle \left((I:J)/I\right)^{\oplus t} \subseteq \ker(\varphi)$.
\end{proof}

\begin{corollary} Let $(R,\m)$ be a $2$-dimensional Cohen-Macaulay local ring, and $J$ be a parameter ideal. For any ideal $I \subseteq R$ we have a short exact sequence
\[
\xymatrix{
0 \ar[r] & \displaystyle \frac{R}{I:J} \ar[r]^-{\varphi} & \displaystyle \left(\frac{R}{I}\right)^{\oplus 2} \ar[r] &  \displaystyle\frac{J}{IJ} \ar[r] & 0.
}
\]
\end{corollary}
\begin{proof}
Apply Lemma \ref{lemma exact sequence} with $t=1$, and observe that
\[
\Tor_1^R(R/I,J) = \ker(\varphi) = \{z \in R/I \mid zJ \subseteq I\} = (I:J)/I. \qedhere
\] 
\end{proof}

\begin{corollary} Let $(R,\m)$ be a $2$-dimensional Cohen-Macaulay local ring, and $J$ be a parameter ideal. For any $t>0$ and any ideal $I \subseteq R$ containing $J$ we have an isomorphism
\[
\left(\frac{R}{I}\right)^{\oplus (t+1)} \cong \frac{J^t}{IJ^t}.
\]
\end{corollary}
\begin{proof}
In the notation of Lemma \ref{lemma exact sequence} it suffices to observe that, since $J \subseteq I$, we have $\IM(\varphi) \subseteq J(R/I)^{\oplus t} = 0$.
\end{proof}

\begin{proposition} \label{estimates} Let $(R,\m)$ be a $2$-dimensional excellent Cohen-Macaulay reduced local ring, and $I$ be an $\m$-primary ideal. With the notation introduced above, we have that $f(n-1) \geq f(n)\geq L_2(I)$ for all $n \gg 0$. 
\end{proposition}
\begin{proof}
Let $J$ be a minimal reduction of $I$. From Lemma \ref{lemma exact sequence}, for every $n,t$ and $q=p^e$ we have an exact sequence
\[
\xymatrix{
0 \ar[r] & \displaystyle \Tor_1^R\left(\frac{R}{(I^{[q]})^n},(J^{[q]})^t\right) \ar[r] & \displaystyle \left(\frac{R}{(I^{[q]})^n}\right)^{\oplus t} \ar[r]^-{\varphi} & \displaystyle \left(\frac{R}{(I^{[q]})^n}\right)^{\oplus (t+1)} \ar[r] &  \displaystyle\frac{(J^{[q]})^t}{(I^{[q]})^n(J^{[q]})^t} \ar[r] & 0.
}
\]
Again by Lemma \ref{lemma exact sequence} we have that
\[
\left(\frac{(I^{[q]})^{n-1}}{(I^{[q]})^n}\right)^{\oplus t} \subseteq \left(\frac{(I^{[q]})^n:J^{[q]}}{(I^{[q]})^n}\right)^{\oplus t} \subseteq \ker(\varphi),
\]
so that we have an exact sequence
\[
\xymatrix{
\displaystyle \left(\frac{R}{(I^{[q]})^{n-1}}\right)^{\oplus t} \ar[r]^-{\varphi} & \displaystyle \left(\frac{R}{(I^{[q]})^n}\right)^{\oplus (t+1)} \ar[r] &  \displaystyle\frac{(J^{[q]})^t}{(I^{[q]})^n(J^{[q]})^t} \ar[r] & 0.
}
\]
Choose $n \geq r(I)$. Since $r(I^{[q]}) \leq r(I)$, we have that $(I^{[q]})^n(J^{[q]})^t = (I^{[q]})^{n+t}$. Counting lengths, we therefore get
\begin{equation} \label{eq}
(t+1)\ell_R\left(\frac{R}{(I^{[q]})^n}\right) - t\ell_R\left(\frac{R}{(I^{[q]})^{n-1}}\right) \leq \ell_R\left(\frac{R}{(I^{[q]})^{n+t}}\right) - \ell_R\left(\frac{R}{(J^{[q]})^t}\right).
\end{equation}
For the moment assume $q$ is fixed, and choose $t \gg 0$ so that $\ell_R(R/(I^{[q]})^{n+t})$ coincides with its Hilbert polynomial. As a function of $t$, we can write
\begin{align*}
\ell_R\left(\frac{R}{(I^{[q]})^{n+t}}\right) & = \ee_0(I^{[q]})\binom{n+t+1}{2} - \ee_1(I^{[q]})(n+t) + \ee_2(I^{[q]}) \\
& =\frac{q^2\ee(I)}{2} t^2 + \left[nq^2\ee(I) + \frac{q^2\ee(I)}{2} - \ee_1(I^{[q]})\right] t + o(t). 
\end{align*}
Since 
\[
\ell_R\left(\frac{R}{(J^{[q]})^t}\right) = \binom{t+1}{2} \ell_R\left(\frac{R}{J^{[q]}}\right) = \frac{q^2\ee(I)}{2}t^2 + \frac{q^2\ee(I)}{2}t,
\]
the inequality (\ref{eq}) becomes
\[
(t+1)\ell_R\left(\frac{R}{(I^{[q]})^n}\right) - t\ell_R\left(\frac{R}{(I^{[q]})^{n-1}}\right) \leq  \left[nq^2 \ee(I) - \ee_1(I^{[q]})\right]t+o(t).
\]
Dividing by $t$ and taking $\lim_{t \to \infty}$ gives 
\[
\ell_R\left(\frac{R}{(I^{[q]})^n}\right) - \ell_R\left(\frac{R}{(I^{[q]})^{n-1}}\right) \leq nq^2 \ee(I) - \ee_1(I^{[q]}).
\]
Now divide by $q^2$ and take $\lim_{q \to \infty}$ to get
\[
\eHK(I^n) - \eHK(I^{n-1}) \leq n\ee(I) - L_1(I).
\]
A direct calculation shows that the quantity on the left-hand side for $n \gg 0$ is $n \ee(I) - L_1(I) + f(n) - f(n-1)$, so that we get $f(n) \leq f(n-1)$. For the last claim it suffices to recall that for all $n \gg 0$ we have
\[
\ee(I)\binom{n+1}{2} -L_1(I)n + f(n) = \eHK(I^n) \geq \widetilde{\eHK}(I^n) = \ee(I)\binom{n+1}{2} -L_1(I)n + L_2(I),
\]
which gives $f(n) \geq L_2(I)$.
\end{proof}

\begin{remark} In the previous proof we use the trivial containment $(I^{[q]})^{n-1} \subseteq (I^{[q]})^n:J^{[q]}$. If $\widetilde{(I^{[q]})^{n-1}} \subseteq (I^{[q]})^n:J^{[q]}$ holds true for all $q \gg 0$, then with the same calculations one gets that $f(n) \leq L_2(I)$. Since the other inequality is always true, equality holds. In particular, if $\widetilde{(I^{[q]})^{n-1}} \subseteq (I^{[q]})^n:J^{[q]}$ holds true for all $q,n \gg 0$, then (Q2) has positive answer for $I$.
\end{remark}

\section{A possible negative answer to (Q2)} \label{Section counterex}
We suspect that (Q2) has negative answer, even in dimension two. Even if at the moment we are not able to fully support this claim, we present a ring which seems to contradict the expectations of (Q2), at least after some experimental calculations performed with Macaulay 2 \cite{M2}. Before, we recall the notion of test ideal. Let $R$ be a ring of characteristic $p>0$. The (finitistic) test ideal of $R$ is 
\[
\tau(R) = \bigcap_I \left(I:_R I^*\right),
\]
where the intersection runs over all ideals $I$ of $R$. If $R$ is F-finite, and $c \in R^\circ$ is such that $R_c$ is regular, then $c$ has a power which belongs to $\tau(R)$; such an element is called a test element. Thus, if $J$ is an ideal which defines the singular locus of $R$, we have that $J \subseteq \sqrt{\tau(R)}$. In particular, if $R$ has an isolated singularity at a maximal ideal $\m$, then $\tau(R)$ is either the whole ring or it is an $\m$-primary ideal. If, in addition, $R$ is F-pure, then $\tau(R)$ is a radical ideal; for instance, see \cite{FedderWatanabe,VassilevCowden}. As a consequence, if $R$ is an F-finite and F-pure isolated singularity at $\m$, then $\tau(R) \supseteq \m$. In addition, if $R$ is not weakly F-regular (that is, there exists an ideal which is not tightly closed) then $\tau(R)=\m$.

\begin{example}
Let $k$ be an F-finite field of characteristic $p \ne 3$, let $R=k\ps{x,y,z}/(x^3+y^3+z^3)$, and $\m=(x,y,z)R$. Note that $R$ has an isolated singularity at $\m$, it is F-finite and not weakly F-regular. If $p \equiv 1 \bmod 3$, then it is also F-pure, and in this case we then have $\tau(R)=\m$. 
In this case, if $J \subseteq I$ are two ideals, in order to show that $I \not\subseteq J^*$ one only has to prove that $\m I \not\subseteq J$. For $p=7$, calculations performed with Macaulay2 seem to suggest that, for all $n$, there exists $q=7^e \gg 0$ such that 
\[
\m \cdot \left((\m^{[q]})^{n+1} :_R \m^{[q]}\right) \not\subseteq (\m^{[q]})^n.
\]
When this is the case, we have that $\widetilde{(\m^{[q]})^n} \not\subseteq \left((\m^{[q]})^n\right)^*$, and by Proposition \ref{equivalence conj} we conclude that $\eHK(\m^n) \ne \widetilde{\eHK}(\m^n)$ for such $n$. If this was indeed true for all $n$ (or even just for infinitely many $n$), then (Q2) would have negative answer for this ring with respect to the ideal $\m$.
\end{example}

\begin{remark} If $R=k\ps{x,y,z}/(x^3+y^3+z^3)$, and ${\rm char}(k)=2$ or $5$, then one can still show that $\tau(R) = \m$ even if $R$ is not F-pure. These choices of smaller characteristics could be helpful for calculations, and perhaps still provide a negative answer to (Q2).
\end{remark}

In \cite{RossiSwanson} it is proved that if $(R,\m)$ is $2$-dimensional Cohen-Macaulay, $I$ is an $\m$-primary ideal and $J=(x,y)$ is  a minimal reduction of $I$ such that $I^{m+1}=JI^m$ and $I^{m+1}:x=I^m$, then $\widetilde{I^n} = I^n$ for all $n \geq m$. Rossi and Swanson also mention that they do not know whether $I^{m+1}=JI^m$ is sufficient to imply that $\widetilde{I^m}=I^m$. We show, with the aid of Macaulay 2, that this is indeed not the case.

\begin{example} Let $R=\FF_2[x,y,z]/(x^3+y^3+z^3)$, let $\m=(x,y,z)R$ and $J=(y,z)$, which is a minimal reduction of $\m$. Note that $r_J(\m) =2$, and therefore $(\m^{[q]})^3 = J^{[q]}(\m^{[q]})^2$ for all $q=2^e$. Calculations performed with Macaulay2 give that the element $x^2y^4z^{13}$ belongs to $(\m^{[8]})^4:(\m^{[8]})^2$, and hence to $\widetilde{(\m^{[8]})^2}$, but not to $(\m^{[8]})^2$. To have a local example, note that all these facts still hold in the localization $R_\m$.
\end{example}

\bibliographystyle{plain}
\bibliography{References}

\end{document}